\documentclass[11pt,reqno]{amsart}

\usepackage{xcolor}
\usepackage{graphicx}
\usepackage[colorlinks=true,allcolors=blue,backref=page]{hyperref}
\usepackage{amsmath,amssymb,amsthm,mathtools}
\usepackage{mathrsfs}
\usepackage[noabbrev,capitalize,nameinlink]{cleveref}
\usepackage{aliascnt}
\usepackage{fullpage}
\usepackage[noadjust]{cite}
\usepackage{bbm}
\usepackage[T1]{fontenc}
\allowdisplaybreaks[1]
\numberwithin{equation}{section}

\crefname{equation}{}{}
\crefname{conjecture}{Conjecture}{Conjectures}
 
\newtheorem{theorem}{Theorem}[section]

\newaliascnt{proposition}{theorem}
\newtheorem{proposition}[proposition]{Proposition}
\aliascntresetthe{proposition}
\crefname{proposition}{Proposition}{Propositions}

\newaliascnt{lemma}{theorem}
\newtheorem{lemma}[lemma]{Lemma}
\aliascntresetthe{lemma}
\crefname{lemma}{Lemma}{Lemmas}

\newaliascnt{corollary}{theorem}

\aliascntresetthe{corollary}
\crefname{corollary}{Corollary}{Corollaries}

\crefname{claim}{Claim}{Claims}
\newtheorem{conjecture}[theorem]{Conjecture}

\theoremstyle{definition}

\theoremstyle{remark}
\newtheorem*{remark}{Remark}

\newcommand{\mc}{\mathcal}

\newcommand{\one}{\mathbbm{1}}

\newcommand{\R}{\mathbb{R}}
\newcommand{\E}{\mathbb{E}}
\newcommand{\PP}{\mathbb{P}}

\renewcommand{\le}{\leqslant}
\renewcommand{\ge}{\geqslant}

\newcommand{\inner}[2]{\langle #1,#2 \rangle}
\DeclarePairedDelimiter{\snorm}{\lVert}{\rVert}
\DeclarePairedDelimiter{\abs}{\lvert}{\rvert}
\DeclarePairedDelimiter{\set}{\{}{\}}
\DeclarePairedDelimiter{\paren}{(}{)}

\DeclareMathOperator{\conv}{conv}
\DeclareMathOperator{\supp}{supp}
\DeclareMathOperator{\vrad}{vrad}

\newcommand{\Sph}{S^{n-1}}

\newcommand{\gauss}{\gamma_n}
\newcommand{\rvLp}[2]{\paren{\E\abs{#1}^{#2}}^{1/#2}}
\newcommand{\margLp}[3]{\paren{\E_{X\sim #1}\abs{\inner{X}{#2}}^{#3}}^{1/#3}}

\title{Dimension-free Gaussian tail estimates for linear functionals on convex bodies}

\author[Letwin]{Brayden Letwin}
\author[Mikulincer]{Dan Mikulincer}
\address{Department of Mathematics, University of Washington, Seattle, WA 98195}
\email{\{letwin,danmiku\}@uw.edu}

\begin{document}

\begin{abstract}
Let $K \subset \mathbb{R}^n$ be a centered convex body of volume one. We prove that there exist absolute constants $c,C > 0$ and an orthonormal set of vectors $\Theta \subset S^{n-1}$ with size $\left|\Theta\right| \ge 9n/10$ such that, if $X$ is a random vector uniformly distributed on $K$, then for all $\theta \in \Theta$ one has
\[
    c\cdot \sqrt{p}\,\left(\mathbb{E} \left|\left\langle X,\theta \right\rangle\right|^2\right)^{1/2}
    \le
    \left(\mathbb{E} \left|\left\langle X,\theta \right\rangle\right|^p\right)^{1/p}
    \le
    C\cdot \sqrt{p}\,\left(\mathbb{E} \left|\left\langle X,\theta \right\rangle\right|^2\right)^{1/2},
\]
where the upper estimate holds for all $p \ge 1$ while the lower bound only holds for $1 \le p \le n$.
\end{abstract}

\maketitle

\section{Introduction}\label{sec:1}
Let $K \subset \R^n$ be a centered convex body of volume one, and let $X$ be a random vector uniformly distributed on $K$. The behavior of one-dimensional marginals $\inner{X}{\theta}$, for $\theta \in \Sph$, has been an important area of research over the past few decades. By the Brunn--Minkowski inequality, one may show that $\inner{\cdot}{\theta}$ has subexponential tails: for all $t \ge 1$ one has
\[
    \PP\paren{\abs{\inner{X}{\theta}} \ge t\cdot \E\abs{\inner{X}{\theta}}} \le \exp(-t/C),
\]
for some absolute constant $C > 0$. Going beyond this subexponential estimate, stronger concentration estimates for linear functionals, and more generally for Lipschitz functions, were obtained in the case where $K$ is uniformly convex and symmetric \cite{GM87, GM83}. In a more general setting, V. Milman posed the question of whether \emph{every} centered convex body $K$ of volume one admits at least one "subgaussian direction". That is, does there exist $\theta \in S^{n-1}$ such that the linear functional $\inner{\cdot}{\theta}$ has the following subgaussian behavior?
\begin{equation*}
    \PP\paren{\abs{\inner{X}{\theta}} \ge t \cdot \E\abs{\inner{X}{\theta}}} \le \exp(-t^2/C^2),
\end{equation*}
for all $t \ge 1$, where $C>0$ is an absolute constant. There has been substantial progress towards this question, with connections and contributions to several neighbouring problems. However, the existence of such a dimension-free subgaussian constant remained an open question. In this work, we resolve V. Milman's question by establishing the existence of a subgaussian direction. In fact, we prove that there exists an orthonormal set of vectors $\Theta \subset \Sph$ with size $\abs{\Theta} \ge 9n/10$ such that $\inner{\cdot}{\theta}$ simultaneously has subgaussian and supergaussian tails. 
\begin{theorem}\label{thm:1.1}
Let $K \subset \R^n$ be a centered convex body of volume one. There exists an orthonormal set of vectors $\Theta \subset S^{n-1}$ with size $\abs{\Theta} \ge 9n/10$, and absolute constants $C,c>0$ such that, if $X$ is a random vector uniformly distributed on $K$, then for all $\theta \in \Theta$ one has
\begin{equation} \label{eq:1.1}
    c\cdot \sqrt{p}\,\paren{\E \abs{\inner{X}{\theta}}^2}^{1/2}
    \le
    \paren{\E \abs{\inner{X}{\theta}}^p}^{1/p}
    \le
    C\cdot \sqrt{p}\,\paren{\E \abs{\inner{X}{\theta}}^2}^{1/2}
\end{equation}
where the upper bound holds for all $p \ge 1$ while the lower bound only holds for $1 \le p \le n$. In particular, this implies for all $\theta \in \Theta$, and other absolute constants $c',C'>0$,
\begin{equation} \label{eq:1.2}
    \exp(-C'\cdot t^2) \le \PP\paren*{\abs{\inner{X}{\theta}} \ge t\cdot \E\abs{\inner{X}{\theta}}} \le 2\exp(-t^2/C'^2),
\end{equation}
where the upper bound holds for all $t \ge 1$ while the lower bound only holds for $1 \le t \le c'\cdot \sqrt{n}$.
\end{theorem}
We will only prove \cref{eq:1.1}. The tail estimate \cref{eq:1.2} is an immediate corollary of Markov's inequality and the Paley--Zygmund inequality; for the standard argument which translates from absolute moments to tails we invite the reader to consult \cite[Section~2.6]{Vershynin18}. The lower estimate in \cref{eq:1.2} cannot, in general, be improved to hold for all $t \ge 1$, and this may be seen by considering the Euclidean ball in $\R^n$, normalized to have volume one. 

The problem of finding a subgaussian direction in a convex body is situated within the study of high-dimensional convex bodies and asymptotic convex geometry, see the monograph \cite{BGVV14}, as well as the books \cite{AGM15,AGM21}, for more information on these topics. The specific question of subgaussian directions has a substantial history. The problem at hand was originally proved in the special cases of $l_p$ balls \cite{BGMN05}, unconditional bodies \cite{BN03a, BN03b}, and zonoids \cite{Paouris03}. Subsequent works have focused on improving the dimension-dependent bounds on the constant $C$ from \cref{thm:1.1}. Klartag, in \cite{Klartag08}, gave the first poly-logarithmic bound (see also \cite{GPP07} for another proof), which was later improved to $C\log^{1/2} n$ in \cite{GPV11}, by Giannopoulos, Paouris, and Valettas. In parallel, many works, such as \cite{DP10,Bourgain03,Paouris05}, have explored the connections between subgaussian directions and other questions in convex geometry, such as Bourgain's slicing conjecture, which was recently resolved in \cite{KL25,Bizeul25}.

For the lower estimates in \cref{thm:1.1}, we also mention the work of Paouris \cite{Paouris12a} (see \cite{klartag2017super} as well), who proved that every convex body $K$ admits a direction with supergaussian lower tails. Our conclusion is closely related in spirit: instead of proving a separate inequality, we prove, for $1 \le p \le n$, 
\[c\cdot \sqrt{p}\,(\E \abs{\inner{X}{\theta}}^2)^{1/2} \le (\E \abs{\inner{X}{\theta}}^p)^{1/p},
\]
for the same orthonormal set $\Theta$ that also satisfies the dimension-free subgaussian upper estimates. 

Several neighboring literatures should also be kept in view. One concerns \emph{typical} marginals of convex bodies, and includes important results such as the central limit theorem for convex bodies \cite{Klartag07a,Klartag07b, EK08}. Another connected line of research concerns the concentration of norms or other functionals, with notable contributions in \cite{Paouris06, Fleury10a, Fleury10b, Pivovarov10, LMS98, FGP07}. Finally, there are also works studying anti-concentration in the form of small-ball probabilities, see for example \cite{Paouris12b, DP10}. Our theorem is adjacent to all three strands, but different in the sense that we aim to understand the tails of specific one-dimensional marginals from a large orthonormal family, which can exhibit atypical behavior. 

With respect to the lines of research mentioned above, a relevant question asks about the typical subgaussian constant of a \emph{random} one-dimensional marginal of a convex body.  Again, there is a long line of works on this problem \cite{Paouris05, GPV12a, GPV12b} with the current best result from \cite{BH15}, building upon \cite{Milman15}, showing that a typical marginal of an isotropic convex body is subgaussian with $C \cdot \log^{2}n$.  They also note that in general this dependence cannot be improved below $\log^{1/2} n.$

We remark that the constant $9/10$ in \cref{thm:1.1} is inconsequential and can be replaced by any fixed $0 \le \beta < 1$. That is, for such a $\beta$, by repeating our argument mutatis mutandis and perhaps increasing the value of the constant $C$, we can ensure $|\Theta| \geq \beta n$.

Finally, we raise the question of whether \cref{thm:1.1} can be strengthened in the sense that one may find an \emph{orthonormal basis} $\{\theta_1, \ldots \theta_n\} \subset S^{n-1}$ satisfying \cref{eq:1.1}. We return to this question in \cref{sec:4}, where we explain why the greedy construction used below does not by itself produce a full basis.

\subsection{Overview of the proof}\label{subsec:1.1}
Our proof of \cref{thm:1.1} follows the approach outlined in \cite{GPP07} and \cite{GPV11} and goes through considering the $\Psi_2$ bodies. For a convex body $K$, $\Psi_2(K)$ is another convex body which encodes information about the subgaussian constants of the marginals of $K$, see \cref{subsec:2.2} for a precise definition. The key point is that the existence of a subgaussian direction is equivalent to finding a point with large norm in the polar body $\Psi_2(K)^\circ$.

Similar to the strategy of \cite{GPV11}, when $K$ is isotropic, we discretize the defining constraints of $\Psi_2(K)^\circ$ by considering the finite set
\[
    \mc D = \set{2^j : 0 \le j,\ 2^j \le c_0\cdot n},
\]
for some absolute constant $c_0 > 0$, and by associating a convex body for every $p \in \mc D$
\begin{equation} \label{eq:1.3}
    A_p
    =
    \{
        y \in \R^n : \margLp{K}{y}{p} \le C\cdot \sqrt{np}\,L_K
    \},
\end{equation}
where $L_K$ is the isotropic constant of $K$.
As we will explain, the support-function constraints defining $A = \cap_{p \in \mc D} A_p$ are a discretized and scaled version of the constraints defining $\Psi_2(K)^\circ$. In light of the above, it is now natural to bound the volume (or the Gaussian measure) of $A$ from below, which will then translate to a bound on $\|y\|$ for some $y \in A$. In \cite{GPV11}, the authors obtain such a bound by analyzing the covering numbers of $\Psi_2(K)$.

For the sharp bound in \cref{thm:1.1}, we replace the covering numbers argument with Gaussian measure arguments. First, for $\gamma_n$ the standard Gaussian measure on $\R^n$, we show that for every $p \in \mc D$,
$$\gamma_n(A_p) \ge e^{-C\cdot p}.$$
This is based on negative moment estimates from \cite{Paouris12b}, which were also used in \cite{GPV12a, GPV11}. Similar to these works, we transfer the negative moment estimates to Gaussian space and apply the Paley--Zygmund inequality. Then we apply the Gaussian correlation inequality, from which we deduce
\begin{equation} \label{eq:1.4}
    \gauss(A)
    \ge
    \prod_{p \in \mc D} \gauss(A_p)
    \ge
    e^{-C'\cdot n}.
\end{equation}

Finding $y \in A$ with $\|y\|$ large is now a matter of excluding smaller sublevel sets. For $p \in \mc D$, put
\[
    B_p
    =
    \set{
        y \in \R^n : \margLp{K}{y}{p} \le \varepsilon\cdot \sqrt{np}\,L_K
    }.
\]

We prove that, for every subspace $F \subset \R^n$ with $\dim F \ge n/10$,
\[
    (A \cap F) \setminus \bigcup_{p \in \mc D} (B_p \cap F) \neq \varnothing.
\]

The point is that the Gaussian measure lower bound on \(A\) forces \(A\cap F\) to have large volume-radius, whereas each \(B_p\cap F\) has volume-radius of order \(\varepsilon \sqrt n\). Taking \(\varepsilon\) sufficiently small therefore prevents the union of the bad pieces from covering \(A\cap F\).

Choosing points in these annuli inside successive orthogonal complements gives an orthonormal set $\Theta \subset \Sph$ with $\abs{\Theta}\ge 9n/10$ such that, for $p \in \mc D$ and $\theta \in \Theta$,
\[
    c\cdot \sqrt{p}\,L_K
    \le
    \margLp{K}{\theta}{p}
    \le
    C\cdot \sqrt{p}\,L_K.
\]

Extending beyond $p \in \mc D$ to all $1 \le p \le n$ is routine and follows from one-dimensional log-concave absolute moment comparison inequalities.
For the upper estimate in \cref{eq:1.1}, we also extend to $p>n$ using the concavity properties of the one-dimensional marginal density, as follows from the Brunn--Minkowski inequality. 

\subsection{Conventions} \label{subsec:1.2}
Throughout the paper, \(c,C,c_0,C_0,c_1,C_1,\ldots>0\) denote absolute constants whose values may change from line to line. When a constant depends on an auxiliary parameter, we indicate that dependence with a subscript, as in \(c_\beta\) or \(C_{c_0}\). We will use this convention without further comment.

\section{Preliminaries}\label{sec:2}

\subsection{Basic notions of convex bodies}\label{subsec:2.1}
If \(K\subset\R^n\) is a convex body, we say that \(K\) is centered if the uniform probability measure on \(K\) has mean zero, that is
\[
    \frac{1}{\abs{K}}\int_K x\,dx=0.
\]
We use the same terminology for random vectors and probability measures: they are centered when their expectation is zero. If \(K\subset \R^n\) is a convex body with \(0\in\operatorname{int}(K)\), we write
\[
    \snorm{x}_K
    =
    \inf\set{t \ge 0:x\in tK}
\]
for its Minkowski functional and
\[
    K^\circ
    =
    \set{y\in \R^n:\inner{x}{y}\le 1\text{ for all }x\in K}
\]
for its polar. Note that for convex bodies, one has $\left(K^{\circ}\right)^{\circ} = K$. For a random variable \(X\), we write its $L^p$ norm as \(\rvLp{X}{p}\). When \(K\subset \R^n\) has volume one, expectations of the form \(\margLp{K}{\theta}{p}\) are taken with respect to the uniform probability measure on \(K\). For \(p\ge1\), the \(L^p\) centroid body \(Z_p(K)\) is the symmetric convex body determined through its polar
\begin{equation} \label{eq:2.1}
        Z_p(K)^\circ
    =
    \set{
        \theta\in\R^n:
        \margLp{K}{\theta}{p}\le1
    }.
\end{equation}

A centered convex body \(K\subset\R^n\) of volume one is called isotropic if there is a number \(L_K>0\) such that
\[
    \int_K \inner{x}{\theta}^2\,dx
    =
    L_K^2\abs{\theta}^2
    \qquad\text{for every }\theta\in\R^n.
\]
Equivalently, the covariance matrix of the uniform measure on \(K\) is \(L_K^2\operatorname{Id}\). The number \(L_K\) is called the isotropic constant of \(K\).

\subsection{\texorpdfstring{$\psi_2$}{psi2} norms, absolute moments, and tails}\label{subsec:2.2}
If $X$ is a real random variable, write
\[
    \snorm{X}_{\psi_2}
    =
    \inf\set{
        t > 0 : \E \exp\paren*{\frac{X^2}{t^2}} \le 2
    }.
\]
for the $\psi_2$ norm of $X$.
For a bounded random variable $X$, the quantities
\[
    \snorm{X}_{\psi_2},
    \qquad
    \inf\set{A>0:\PP\paren{\abs X\ge t}\le 2e^{-t^2/A^2}\ \text{for all }t\ge0},
    \qquad
    \sup_{p\ge1}\frac{\rvLp{X}{p}}{\sqrt p}
\]
are equivalent up to absolute constants; in particular, 
\begin{equation} \label{eq:2.2}
    c\cdot \sup_{p \ge 1} \frac{\rvLp{X}{p}}{\sqrt{p}}
    \le
    \snorm{X}_{\psi_2}
    \le
    C\cdot \sup_{p \ge 1} \frac{\rvLp{X}{p}}{\sqrt{p}}.
\end{equation}
Thus a uniform bound on $\rvLp{X}{p}/\sqrt p$ is equivalent, up to
absolute constants, to a $\psi_2$ bound, and hence to subgaussian tail decay.
See, for example, \cite[Section~2.6]{Vershynin18}.

For a centered convex body \(K\subset\R^n\) of volume one, we package these absolute moment constraints into the body
\[
    \Psi_2(K)
    =
    \conv\paren*{
        \bigcup_{1\le p\le n}\frac{1}{\sqrt p}Z_p(K)
    }.
\]
Equivalently,
\[
    \Psi_2(K)^\circ
    =
    \bigcap_{1\le p\le n}\sqrt p\,Z_p(K)^\circ
    =
    \set{
        y\in\R^n:
        \margLp{K}{y}{p}
        \le
        \sqrt p
        \text{ for all }1\le p\le n
    }.
\]
Thus, if \(y\in\Psi_2(K)^\circ\) has large Euclidean norm and \(\theta=y/\abs y\), then the $L^p$ norms of the marginal \(\inner{\cdot}{\theta}\) are small for every $1\le p \le n$. In the proof below, we use a dilated version of this polar body and discretize $1 \le p \le n$ by instead considering $\mc{D}$ as described above through \cref{eq:1.3}.

\subsection{Linear invariance of centroid bodies}\label{subsec:2.3}
We will use the following elementary linear invariance principle. If $K \subset \R^n$ is a centered convex body of volume one and $T \colon \R^n \to \R^n$ is linear with $\det T = 1$, then for every $p \ge 1$ and every $\theta \in \R^n$,
\[
    \margLp{TK}{\theta}{p}
    =
    \margLp{K}{T^\top \theta}{p}.
\]
Consequently,
\[
    Z_p(TK) = T Z_p(K).
\]
The first identity is a change of variables. The second follows because
\[
    y\in Z_p(TK)^\circ
    \quad\Longleftrightarrow\quad
    T^\top y\in Z_p(K)^\circ
    \quad\Longleftrightarrow\quad
    y\in (T Z_p(K))^\circ.
\]
Thus the centroid-body constructions below are compatible with volume-preserving linear changes of variables. Throughout the remainder of the paper, we will require the existence of a linear map $T : \R^n \to \R^n$ which places $TK$ in isotropic position. \cref{lem:2.1} accomplishes this for us.

\begin{lemma}\label{lem:2.1}
Let $K \subset \R^n$ be a centered convex body of volume one, and let $\Sigma_K$ be the covariance matrix of the uniform measure on $K$. Set
\[
    T = (\det \Sigma_K)^{1/(2n)} \Sigma_K^{-1/2}.
\]
Then $\det T = 1$ and $TK$ is isotropic. Moreover, if $u \in \R^n$ and $a = Tu$ then for every $p \ge 1$,
\[
    \margLp{K}{a}{p}
    =
    \margLp{TK}{u}{p}.
\]
In particular,
\[
    \frac{\snorm{\inner{\cdot}{a}}_{\psi_2(K)}}{\margLp{K}{a}{2}}
    =
    \frac{\snorm{\inner{\cdot}{u}}_{\psi_2(TK)}}{\margLp{TK}{u}{2}}.
\]
\end{lemma}

\begin{proof}
The determinant identity follows from
\[
    \det T
    =
    (\det\Sigma_K)^{1/2}\det(\Sigma_K^{-1/2})
    =
    1.
\]
If \(X\) is uniformly distributed on \(K\), then \(TX\) is uniformly distributed on \(TK\), and
\[
    \operatorname{Cov}(TX)
    =
    T\Sigma_KT
    =
    (\det\Sigma_K)^{1/n}\operatorname{Id}.
\]
Thus \(TK\) is isotropic. Since \(T\) is symmetric, the change of variables \(y=Tx\) gives
\[
    \margLp{TK}{u}{p}
    =
    \margLp{K}{Tu}{p}
    =
    \margLp{K}{a}{p}.
\]
The \(L^2\) and \(\psi_2\) ratio identities follow from the equality of the one-dimensional random variables
\[
    \inner{TX}{u}
    =
    \inner{X}{T^\top u}
    =
    \inner{X}{a},
\]
where \(X\) is uniformly distributed on \(K\).
\end{proof}
The upshot of \cref{lem:2.1} is that it will suffice to prove \cref{thm:1.1} for isotropic bodies.

\subsection{A one-dimensional absolute moment comparison inequality}\label{subsec:2.4}
We use the following comparison to pass between scales. It is the one-dimensional log-concave absolute moment comparison, equivalently the centroid-body inclusion \(Z_q(K)\subset C\cdot(q/p)Z_p(K)\) appearing in \cite[Equation~2.8]{GPV11}.

\begin{lemma}\label{lem:2.2}
There exists an absolute constant \(C>0\) such that the following holds. Let $K \subset \R^n$ be a convex body of volume one. Then, for every $\theta \in \R^n$ and every $1 \le p \le q$,
\[
    \margLp{K}{\theta}{q}
    \le
    C\cdot \frac{q}{p}
    \margLp{K}{\theta}{p}.
\]
In particular, if $q \le 2p$, then
\[
    \frac{\margLp{K}{\theta}{q}}{\sqrt{q}}
    \le
    2C\cdot \frac{\margLp{K}{\theta}{p}}{\sqrt{p}}.
\]
\end{lemma}
\begin{proof}
The first displayed inequality is precisely the one-dimensional form of \cite[Equation~2.8]{GPV11}. The second follows from \(q/p\le2\) and \(p\le q\).
\end{proof}

\subsection{Negative moments of centroid-body norms}\label{subsec:2.5}
In light of \cref{eq:2.2}, we will need to bound the higher absolute moments of one-dimensional marginals of $K$. The definition of $Z_p(K)^\circ$ suggests that it will be useful to instead understand negative moments of the norm whose unit ball is \(Z_p(K)^\circ\). This section aims to understand the latter quantity.

We first introduce some notation. Let $\sigma$ be the normalized Haar measure on $\Sph$, and let $\gauss$ be the standard Gaussian measure on $\R^n$. For a convex body $K \subset \R^n$ with \(0\in\operatorname{int}(K)\) and $q > 0$, write
\[
    W_{-q}(K)
    =
    \left(
        \int_{\Sph} \snorm{u}_{K^\circ}^{-q}\,d\sigma(u)
   \right)^{-1/q}.
\]
For $-n < q < \infty$ with $q \neq 0$, write
\[
    I_q(K)
    =
    \left(\int_K \abs{x}^q\,dx\right)^{1/q}.
\]
We will also need the corresponding Gaussian negative moment of such a norm. For a convex body \(K\subset\R^n\) with \(0\in\operatorname{int}(K)\), set
\[
    G_{-q}(K)
    =
    \left(
        \int_{\R^n} \snorm{y}_{K^\circ}^{-q}\,d\gauss(y)
   \right)^{-1/q}.
\]

The next lemma connects the different notions through polar integration.
\begin{lemma}\label{lem:2.3}
If $K\subset \R^n$ is a convex body with \(0\in\operatorname{int}(K)\) and $0<q\le n/2$, then
\[
    c\cdot \sqrt{n}\,W_{-q}(K)
    \le
    G_{-q}(K)
    \le
    C\cdot \sqrt{n}\,W_{-q}(K).
\]
\end{lemma}

\begin{proof}
Writing \(y=tu\) with \(u\in\Sph\) and \(t>0\), we have
\[
    \int_{\R^n}\snorm{y}_{K^\circ}^{-q}\,d\gauss(y)
    =
    (2\pi)^{-n/2}
    \left(\int_0^\infty t^{n-q-1}e^{-t^2/2}\,dt\right)
    |S^{n-1}|
    \int_{\Sph}\snorm{u}_{K^\circ}^{-q}\,d\sigma(u).
\]
Thus, applying standard formulas for radial Gaussian moments,
\[
    G_{-q}(K)
    =
    \left(
        \frac{2^{-q/2}\Gamma((n-q)/2)}{\Gamma(n/2)}
    \right)^{-1/q}
    W_{-q}(K).
\]
By Wendel's inequality, applied from both sides to the Gamma ratio, the prefactor above is comparable to \(\sqrt n\) with absolute constants whenever \(0<q\le n/2\).
\end{proof}
The following proposition is a consequence of \cite[Equations~(5.5) and~(5.6)]{GPV11}. We provide a proof, based on the recent small-ball estimates of \cite{Bizeul25}, for completeness. 
\begin{proposition}\label{prop:2.4}
There exist absolute constants $c_0 \in (0,1/4]$ and $c,C > 0$ such that the following holds. Let $K \subset \R^n$ be a isotropic convex body. Then for every $p$ with
\[
    1 \le p \le c_0\cdot n,
\]
one has
\[
    c\cdot \sqrt{p}\,L_K
    \le
    W_{-p}(Z_p(K))
    \le
    C\cdot \sqrt{p}\,L_K,
\]
as well as,
\[
    c\cdot \sqrt{p}\,L_K
    \le
    W_{-2p}(Z_p(K))
    \le
    C\cdot \sqrt{p}\,L_K.
\]
\end{proposition}

\begin{proof}
Let $X$ be uniformly distributed on $K$, and set
\[
    \widetilde X = \frac{X}{L_K}.
\]
Then $\widetilde X$ is isotropic in the probabilistic normalization:
\[
    \E \widetilde X = 0,
    \qquad
    \operatorname{Cov}(\widetilde X) = \operatorname{Id}.
\]
Setting \(y=0\) in the small-ball estimate \cite[Theorem~1]{Bizeul25} yields absolute constants $\varepsilon_0,\beta > 0$ such that, for every $0 < \varepsilon < \varepsilon_0$,
\[
    \PP\paren{\abs{\widetilde X} \le \varepsilon \sqrt{n}}
    \le
    \varepsilon^{\beta n}.
\]
Equivalently, for every $0 < \varepsilon < \varepsilon_0$,
\[
    \PP\paren{\abs{X} \le \varepsilon \sqrt{n}\,L_K}
    \le
    \varepsilon^{\beta n}.
\]
We choose the constant \(c_0\) in the statement small enough so that
\[
    c_0 \le \min\set*{\frac{\beta}{4},\frac14}.
\]

We first turn this small-ball estimate into a lower bound for the negative moments of the Euclidean norm.
Fix a $q$ with
\[
    1 \le q \le 2\cdot c_0\cdot n.
\]
Then
\[
    c\cdot \sqrt{n}\,L_K
    \le
    I_{-q}(K)
    \le
    \sqrt{n}\,L_K.
\]
The upper estimate is immediate from Jensen's inequality and the following second absolute moment bound: 
\[
    I_{-q}(K)
    \le
    I_2(K)
    =
    \left(\int_K \abs{x}^2\,dx\right)^{1/2}
    =
    \sqrt{n}\,L_K.
\]
For the lower estimate, use the layer-cake formula:
\[
    \E \abs{X}^{-q}
    =
    q \int_0^\infty t^{-q-1}\PP\paren{\abs{X} \le t}\,dt.
\]
Split the integral at
\[
    t_0 = \varepsilon_0 \sqrt{n}\,L_K.
\]
Using the small-ball estimate on $(0,t_0)$ and the bound $\PP\paren{\abs{X} \le t} \le 1$ on $(t_0,\infty)$ gives
\[
    \E \abs{X}^{-q}
    \le
    q \int_0^{t_0} t^{-q-1} \paren*{\frac{t}{\sqrt{n}\,L_K}}^{\beta n}\,dt
    +
    \int_{t_0}^\infty q t^{-q-1}\,dt.
\]
Since $q \le 2\cdot c_0\cdot n \le \beta\cdot n/2$, this gives
\[
    \E \abs{X}^{-q}
    \le
    \paren{\sqrt{n}\,L_K}^{-q}
    \left(
        \frac{q \varepsilon_0^{\beta n-q}}{\beta n-q}
        +
        \varepsilon_0^{-q}
    \right)
    \le
    \left(\frac{C}{\sqrt{n}\,L_K}\right)^q.
\]
Hence
\[
    I_{-q}(K) \ge c\cdot \sqrt{n}\,L_K,
\]
which proves the displayed lower bound for $I_{-q}(K)$.

We next use the negative moment comparison between negative moments of the Euclidean norm and centroid bodies, in the form given in \cite[Equation~5.2]{GPV11}; see also \cite{Paouris12b}. For every $q$ with $1 \le q < n$,
\[
    c\cdot \sqrt{\frac{n}{q}}\,W_{-q}(Z_q(K))
    \le
    I_{-q}(K)
    \le
    C\cdot \sqrt{\frac{n}{q}}\,W_{-q}(Z_q(K));
\]
Applying this with $q=p$ shows
\[
    c\cdot \sqrt{p}\,L_K
    \le
    W_{-p}(Z_p(K))
    \le
    C\cdot \sqrt{p}\,L_K.
\]
Applying the same argument with $q=2p$ gives
\[
    c\cdot \sqrt{2p}\,L_K
    \le
    W_{-2p}(Z_{2p}(K))
    \le
    C\cdot \sqrt{2p}\,L_K.
\]
By \cref{lem:2.2},
\[
    Z_{2p}(K) \subset C\cdot Z_p(K),
\]
and therefore
\[
    W_{-2p}(Z_p(K))
    \ge
    c\cdot W_{-2p}(Z_{2p}(K))
    \ge
    c\cdot \sqrt{p}\,L_K.
\]
On the other hand, the map $q \mapsto W_{-q}(K)$ is decreasing for every fixed convex body $K$ with \(0\in\operatorname{int}(K)\), because it is the inverse $L^q$ norm of the random variable $\snorm{U}_{K^\circ}^{-1}$, where $U$ is uniformly distributed on $\Sph$. Hence
\[
    W_{-2p}(Z_p(K))
    \le
    W_{-p}(Z_p(K))
    \le
    C\cdot \sqrt{p}\,L_K.
\]
The displayed estimates follow.
\end{proof}

\subsection{Gaussian correlation}\label{subsec:2.6}
The last component we require, which will help us implement \cref{eq:1.4}, is the Gaussian correlation inequality.
\begin{proposition}[\textnormal{\cite[Theorem~1]{Royen14}}]\label{prop:2.5}
Let $\gauss$ be the standard Gaussian measure on $\R^n$. If $A,B \subset \R^n$ are symmetric convex Borel sets, then
\[
    \gauss(A \cap B) \ge \gauss(A)\gauss(B).
\]
Consequently, for any finite family $A_1,\dots,A_m$ of symmetric convex sets,
\[
    \gauss\paren{\bigcap_{j=1}^m A_j}
    \ge
    \prod_{j=1}^m \gauss(A_j).
\]
\end{proposition}
See also the exposition of Lata{\l}a and Matlak \cite{LM17}. In the proof below we apply this to the convex sets \(A_p\) from \cref{eq:1.3}, which are symmetric because \(Z_p(K)\) is symmetric and \(A_p\) is a dilate of \(Z_p(K)^\circ\).

\section{\texorpdfstring{Proof of \cref{thm:1.1}}{Proof of Theorem 1.1}}\label{sec:3}

By \cref{lem:2.1}, it is enough for the main construction to work in isotropic position. Throughout \cref{subsec:3.1,subsec:3.2,subsec:3.3,subsec:3.4}, \(K\) denotes a centered isotropic convex body of volume one.

\subsection{A fixed scale estimate}\label{subsec:3.1}
Fix $1 \le p \le c_0\cdot n$, where $c_0$ is the constant from \cref{prop:2.4}, and let $C_0$ be some large constant. Define
\begin{equation}\label{eq:3.1}
    A_p
    =
    \set{
        y \in \R^n : \margLp{K}{y}{p} \le C_0\cdot \sqrt{np}\,L_K
    }.
\end{equation}
Equivalently,
\[
    A_p = C_0\cdot \sqrt{np}\,L_K\,Z_p(K)^\circ,
\]
and so it is clear that $A_p$ is symmetric and convex.

\begin{lemma}\label{lem:3.1}
For every $1 \le p \le c_0\cdot n$, if $C_0$ from \cref{eq:3.1} is sufficiently large, then
\[
    \gauss(A_p) \ge e^{-C'\cdot p}.
\]
Here \(C'>0\) is an absolute constant independent of \(p,n\), and \(K\).
\end{lemma}

\begin{proof}
For $y \in \R^n$, recall the definition of the centroid body \eqref{eq:2.1}, and set
\[
    N_p(y) = \margLp{K}{y}{p} = \|y\|_{Z_p(K)^\circ}.
\]
We will apply the Paley--Zygmund inequality to a negative power of $N_p$. To this end, set
\[
    a_p = G_{-p}(Z_p(K)),
    \qquad
    b_p = G_{-2p}(Z_p(K)).
\]
Since $c_0 \le 1/4$ and $p \le c_0\cdot n$, both $p$ and $2p$ are at most $n/2$. Hence \cref{lem:2.3} and \cref{prop:2.4} give the two-sided estimates
\[
    c\cdot \sqrt{np}\,L_K
    \le
    a_p
    \le
    C\cdot \sqrt{np}\,L_K,
\]
\[
    c\cdot \sqrt{np}\,L_K
    \le
    b_p
    \le
    C\cdot \sqrt{np}\,L_K.
\]
In particular, there is an absolute constant $\alpha \ge 1$ such that
\[
    a_p \le \alpha b_p.
\]
Now, let $Y \sim \gamma_n$, and apply the Paley--Zygmund inequality to the non-negative random variable $\xi = N_p(Y)^{-p}$, in a similar fashion to \cite{GPV11}. Then
\[
    \E \xi = a_p^{-p},
    \qquad
    \E \xi^2 = b_p^{-2p}.
\]
Since $a_p \le \alpha b_p$,
\[
    \E \xi^2
    =
    b_p^{-2p}
    \le
    \alpha^{2p} a_p^{-2p}
    =
    \alpha^{2p} (\E \xi)^2.
\]
Hence
\[
    \PP\Big(\xi \ge \frac{1}{2}\E \xi\Big)
    \ge
    \frac{1}{4}\alpha^{-2p}
    \ge
    e^{-C'\cdot p}.
\]
The event $\{\xi \ge \frac{1}{2}a_p^{-p}\}$ is the same as
$\{
    N_p(Y) \le 2^{1/p} a_p \le 2 a_p.
\}$
Since $a_p \le C\cdot \sqrt{np}\,L_K$, as long as $C_0$ is sufficiently large, we obtain
\[
    \gauss\set{y : \margLp{K}{y}{p} \le C_0\cdot \sqrt{np}\,L_K}
    \ge
    e^{-C'\cdot p}.
\]
This proves the lemma.
\end{proof}

\subsection{Intersecting the fixed scale estimates}\label{subsec:3.2}
We assume
\[
    c_0\cdot n \ge 1.
\]
The other case can be absorbed into the absolute constants.
Let
\[
    \mc D = \set{2^j : 0 \le j,\ 2^j \le c_0\cdot n}.
\]
We may decrease the absolute constant \(c_0\) from \cref{prop:2.4} when needed below; the estimates of \cref{prop:2.4} remain valid for the smaller range. 
For each $p \in \mc D$, let $A_p$ be the set in \cref{eq:3.1}. Set
\begin{equation}\label{eq:3.2}
    A = \bigcap_{p \in \mc D} A_p.
\end{equation}

\begin{lemma}\label{lem:3.2}
One has
\[
    \gauss(A) \ge e^{-2C'\cdot c_0\cdot n}.
\]
\end{lemma}

\begin{proof}
Each $A_p$ is symmetric and convex, so by \cref{prop:2.5},
\[
    \gauss(A)
    =
    \gauss\paren{\bigcap_{p \in \mc D} A_p}
    \ge
    \prod_{p \in \mc D} \gauss(A_p).
\]
Using \cref{lem:3.1},
\[
    \prod_{p \in \mc D} \gauss(A_p)
    \ge
    \prod_{p \in \mc D} e^{-C'\cdot p}
    =
    \exp\paren{-C' \sum_{p \in \mc D} p}.
\]
Since $\mc D$ consists of powers of two and its largest element is at most $c_0\cdot n$,
\[
    \sum_{p \in \mc D} p \le 2\cdot c_0\cdot n.
\]
Hence
\[
    \gauss(A) \ge e^{-2\cdot C'\cdot c_0\cdot n}.\qedhere
\]
\end{proof}
\cref{lem:3.2} implies the existence of a vector $y \in A$ with upper bounds on \(\margLp{K}{y}{p}\) for every $p \in \mc D$. However, \(y\) could be very small; in particular \(0\in A\). To obtain meaningful bounds, we remove a family of small sublevel sets from \(A\). The volume bound from \cref{lem:3.2} will show that, after this truncation, enough of \(A\) remains to build a large orthonormal set.

To implement this strategy, fix a small $\varepsilon > 0$, and for $p \in \mc D$, define
\begin{equation}\label{eq:3.3}
    B_p
    =
    \set{
        y \in \R^n : \margLp{K}{y}{p} \le \varepsilon\cdot \sqrt{np}\,L_K
    }.
\end{equation}
Write
\begin{equation}\label{eq:3.4}
    K_p = \frac{1}{\sqrt{p}\,L_K} Z_p(K),
\end{equation}
so that $B_p = \varepsilon\cdot \sqrt{n}\,K_p^\circ$.

\begin{lemma}\label{lem:3.3}
Let $\beta \in (0,1)$. If $c_0$ is sufficiently small, then one can choose $\varepsilon$ in \cref{eq:3.3} sufficiently small such that for every subspace $F \subset \R^n$ with $\dim F \ge \beta n$ one has
\[
    (A \cap F) \setminus \bigcup_{p \in \mc D} (B_p \cap F) \neq \varnothing.
\]
\end{lemma}

\begin{proof}
Let $F \subset \R^n$ be a subspace of dimension $d \ge \beta n$. Write $P_F$ for the orthogonal projection onto $F$, let $\gamma_F$ denote the standard Gaussian measure on $F$, and set
\[
    B_2^F = B_2^n \cap F,
\]
for the unit Euclidean ball in $F$. For a measurable set $C \subset F$, consider the volume-radius
\[
    \vrad_F(C) = \paren*{\frac{|C|}{|B_2^F|}}^{1/d}.
\]
We first compare the section of $A$ with the full Gaussian measure:
\[
    \gamma_F(A \cap F) \ge \gauss(A).
\]
Indeed, decompose $\R^n = F \oplus F^\perp$ and, for $z \in F^\perp$, define
\[
    \varphi(z)
    =
    \int_F \one_A(x+z) e^{-\abs{x}^2/2}\,dx.
\]
The function
\[
    (x,z) \mapsto \one_A(x+z)\exp\paren*{-\frac{\abs{x}^2}{2}}
\]
is even and log-concave on $F \oplus F^\perp$. By the Pr\'ekopa--Leindler inequality, $\varphi$ is even and log-concave on $F^\perp$, hence maximized at the origin. Thus
\[
    \gauss(A)
    =
    (2\pi)^{-n/2}
    \int_{F^\perp} e^{-\abs{z}^2/2} \varphi(z)\,dz
    \le
    (2\pi)^{-d/2} \varphi(0)
    =
    \gamma_F(A \cap F).
\]
Since the Gaussian density on $F$ is bounded above by $(2\pi)^{-d/2}$, \cref{lem:3.2} and the previous inequality give
\[
    |A \cap F| \ge (2\pi)^{d/2}\gauss(A)\ge(2\pi)^{d/2}  e^{-2C'\cdot c_0\cdot n},
\]
and therefore
\[
    \vrad_F(A \cap F)
    \ge
    c\cdot \sqrt{d}\,\exp\paren{-2C'\cdot c_0\cdot n/d}.
\]
Since $d \ge \beta n$, choosing $c_0 > 0$ sufficiently small yields
\[
    \vrad_F(A \cap F) \ge c_\beta\cdot \sqrt{d} \ge c'_\beta\cdot \sqrt{n}.
\]

We next bound the volume-radius of the sections \(B_p\cap F\). The normalization in \cref{eq:3.4} is harmless because it is only a homothety: \(P_FK_p=(\sqrt p\,L_K)^{-1}P_FZ_p(K)\), and volume-radius scales linearly under homotheties. For every $p \in \mc D$, one has
\[
    B_p \cap F
    =
    \varepsilon\cdot \sqrt{n}\,\paren{P_F K_p}^\circ,
\]
where the polar is taken inside $F$. Let $\mu_F$ be the push-forward of the uniform probability measure on $K$ under $P_F$. Then $\mu_F$ is a log-concave probability measure on $F$ with covariance matrix $L_K^2 \operatorname{Id}_F$, and
\[
    (P_F Z_p(K))^\circ
    =
    Z_p(K)^\circ\cap F
    =
    Z_p(\mu_F)^\circ\footnote{The set $Z_q(\mu)$ is defined analogously to $Z_q(K)$ with $\mu$ being an arbitrary probability measure instead of the uniform measure on a convex body $K$.},
\]
where the polars are taken inside \(F\). Indeed, for \(\theta\in F\), the condition
\(\theta\in Z_p(K)^\circ\) is exactly
\(\margLp{K}{\theta}{p}\le1\), which is the same as
\(\margLp{\mu_F}{\theta}{p}\le1\). Thus $P_F Z_p(K)=Z_p(\mu_F)$, in the same spirit as the identities from \cref{subsec:2.3}.
One has the following volume-radius lower bound for \(L^p\)-centroid bodies, which holds for every isotropic (with an identity covariance) log-concave probability measures $\mu$ in \(\R^d\) \cite[Equation~(3.15)]{GSTV15}:
\[
    \abs{Z_p(\mu)}^{1/d}
    \ge
    \tilde{c} \cdot \sqrt{\frac{p}{d}}\,L_\mu^{-1},
    \qquad \text{for all } 1\le p\le d,
\]
where \(L_\mu\) is the isotropic constant of \(\mu\), and volume is taken in the ambient \(d\)-dimensional space. Equivalently, if $B_2^F$ is the radius one Euclidean ball in $F$ then since \(\abs{B_2^F}^{1/d}\) is of order \(d^{-1/2}\) we see,
\[
    \vrad_F(Z_p(\mu))
    \ge
    c\cdot \sqrt p\,L_\mu^{-1}.
\]
Apply this to the push-forward of \(\mu_F\) under the homothety \(x\mapsto L_K^{-1}x\). This is an isotropic log-concave probability measure on \(F\), %and by the boundedness of isotropic constants \cite{KL25,Bizeul25}, its isotropic constant is at most an absolute constant. Since centroid bodies commute with homotheties
and since \(Z_p(\mu_F)=P_FZ_p(K)\), the preceding bound gives
\[
    \vrad_F\paren{L_K^{-1}P_FZ_p(K)}
    \ge
    c\cdot \sqrt p L_{\mu_F}^{-1}
\]
whenever $p \le d$. Since
\[
    p \le c_0\cdot n \le \frac{c_0}{\beta}\cdot d,
\]
after decreasing $c_0$ again if necessary, this applies to every $p \in \mc D$. Therefore
\[
    \vrad_F\paren{P_FZ_p(K)}
    \ge
    c\cdot \sqrt p\,L_KL_{\mu_F}^{-1} \geq c'\sqrt{p}L_K,
\]
where in the second inequality we have bounded $L_\mu$ by some absolute constant, as in \cite{KL25,Bizeul25}.
Hence, with $K_p$ as in \cref{eq:3.4}, 
\[
    \vrad_F(P_F K_p)
    =
    \frac{\vrad_F(P_F Z_p(K))}{\sqrt{p}\,L_K}
    \ge
    c'.
\]
Since \(P_FK_p\) is symmetric, the Blaschke--Santal\'o inequality in $F$ implies
\[
    \vrad_F\paren{(P_F K_p)^\circ} \le \frac{1}{c'}.
\]
Thus
\[
    \vrad_F(B_p \cap F)
    =
    \varepsilon\cdot \sqrt{n}\,\vrad_F\paren{(P_F K_p)^\circ}
    \le
    \frac{1}{c'}\cdot \varepsilon\cdot \sqrt{n}.
\]
It follows that
\[
    \frac{|B_p \cap F|}{|A \cap F|}
    =
    \paren*{\frac{\vrad_F(B_p \cap F)}{\vrad_F(A \cap F)}}^d
    \le
    (C_\beta \cdot \varepsilon)^d.
\]
Since $|\mc D| \le n$, choosing $\varepsilon > 0$ sufficiently small yields
\[
    \sum_{p \in \mc D} |B_p \cap F|
    \le
    |\mc D| (C_\beta \cdot \varepsilon)^d |A \cap F|
    <
    |A \cap F|,
\]
and we can conclude with
\[
    (A \cap F) \setminus \bigcup_{p \in \mc D} (B_p \cap F) \neq \varnothing.\qedhere
\]
\end{proof}
\begin{remark}
    In the proof of \cref{lem:3.3}, we have used the resolution of the slicing conjecture to bound the term $\frac{L_K}{L_{\mu_F}}$. This seems necessary to control the isotropic constant after projecting $K$ onto $F$. Let us point out, though, that if one takes $F = \R^n$, then there is no projection and the terms end up canceling each other. In light of the next lemma, this choice of $F$ corresponds to choosing a single subgaussian direction, rather than some larger orthonormal set.
\end{remark}
We now use \cref{lem:3.3} iteratively to choose many directions satisfying the absolute moment bounds. We state the selection in a form that also keeps track of orthogonality after an arbitrary invertible linear map, which will be useful when we pass back from isotropic position.
\begin{lemma}\label{lem:3.4}
Let \(T:\R^n\to\R^n\) be an invertible linear map. There exists a set $\Theta=\{\theta_1,\dots,\theta_m\}\subset\Sph$, with \(m=\lceil 9n/10\rceil\), such that \(T\theta_1,\dots,T\theta_m\) are pairwise orthogonal and, for every $\theta\in\Theta$ and every $p\in\mc D$,
\begin{equation}\label{eq:3.5}
    c_1\cdot \sqrt{p}\,L_K
    \le
    \margLp{K}{\theta}{p}
    \le
    C_1\cdot \sqrt{p}\,L_K.
\end{equation}
\end{lemma}

\begin{proof}
Set
\[
    m = \left\lceil \frac{9n}{10} \right\rceil.
\]
We choose the directions recursively. Starting with $F_0 = \R^n$, choose vectors $y_1,\dots,y_m$ by requiring
\[
    y_{j+1}
    \in
    (A \cap F_j) \setminus \bigcup_{p \in \mc D} (B_p \cap F_j),
    \qquad
    F_j
    =
    \set{
        y\in\R^n:
        \inner{Ty}{Ty_i}=0
        \text{ for }1\le i\le j
    },
\]
for $0 \le j < m$. By construction, \(Ty_1,\dots,Ty_j\) are nonzero and pairwise orthogonal. Hence \(y_1,\dots,y_j\) are linearly independent, and since \(T^\top T\) is invertible, so are \(T^\top T y_1,\dots,T^\top T y_j\). Therefore \(\dim F_j=n-j\), and \cref{lem:3.3} applies throughout the construction since
\[
    \dim F_j = n-j \ge \frac{n}{10}
\]
for all $0 \le j < m$.

We next control the lengths of the chosen vectors. Since $1 \in \mc D$, the case $p=1$ gives for each $1 \le i \le m$,
\[
    \varepsilon\cdot \sqrt{n}\,L_K
    \le
    \E_{X \sim K} \abs{\inner{X}{y_i}}
    \le
    C_0\cdot \sqrt{n}\,L_K.
\]
Since $K$ is isotropic, \cref{lem:2.2} with \(p=1\) and \(q=2\) gives
\[
    \E_{X \sim K}\abs{\inner{X}{y_i}}
    \le
    \margLp{K}{y_i}{2}
    =
    L_K \abs{y_i}
    \le
    C\cdot \E_{X \sim K}\abs{\inner{X}{y_i}}.
\]
Hence
\[
    c\cdot \sqrt{n} \le \abs{y_i} \le C\cdot \sqrt{n}.
\]

Set
\[
    \theta_i = \frac{y_i}{\abs{y_i}} \in \Sph,
    \qquad
    \Theta = \set{\theta_1,\dots,\theta_m}.
\]
Then \(T\theta_1,\dots,T\theta_m\) are pairwise orthogonal.
For each $1 \le i \le m$ and every $p \in \mc D$,
\[
    \frac{\varepsilon\cdot \sqrt{np}\,L_K}{C\cdot \sqrt{n}}
    \le
    \margLp{K}{\theta_i}{p}
    =
    \frac{\margLp{K}{y_i}{p}}{\abs{y_i}}
    \le
    \frac{C_0\cdot \sqrt{np}\,L_K}{c\cdot \sqrt{n}}.
\]
This is \cref{eq:3.5}.
\end{proof}

\subsection{From scale estimates to all $L^p$ norms}\label{subsec:3.3}
Our next step is to extend the bound in \cref{eq:3.5} from $p \in \mc D$ to arbitrary $q \in [1,n]$. This will be a consequence of the one-dimensional comparison inequalities from \cref{lem:2.2}.

\begin{lemma}\label{lem:3.5}
Let \(\theta\in\Sph\) satisfy \cref{eq:3.5} for every \(p\in\mc D\). Then, for every $1 \le q \le n$,
\begin{equation}\label{eq:3.6}
    c\cdot L_K
    \le
    \frac{\margLp{K}{\theta}{q}}{\sqrt{q}}
    \le
    C\cdot L_K.
\end{equation}
\end{lemma}

\begin{proof}
Let $X$ be uniformly distributed on $K$, and set
\[
    X_\theta = \inner{X}{\theta}.
\]
Since $K$ is convex, $X_\theta$ is a one-dimensional log-concave random variable.

Suppose first that $1 \le q \le c_0\cdot n$. Choose $p \in \mc D$ so that
\[
    p \le q \le 2p.
\]
By \cref{lem:2.2} and \cref{eq:3.5},
\[
    \rvLp{X_\theta}{q}
    \le
    C\cdot \frac{q}{p}\cdot \rvLp{X_\theta}{p}
    \le
   2 C\cdot \rvLp{X_\theta}{p}
    \le
    C'\cdot \sqrt{p}\,L_K
    \le
    C'\cdot \sqrt{q}\,L_K.
\]
Thus
\[
    \frac{\rvLp{X_\theta}{q}}{\sqrt{q}}
    \le
    C'\cdot L_K.
\]
Monotonicity of $L^p$ norms and \cref{eq:3.5} give
\[
    \rvLp{X_\theta}{q}
    \ge
    \rvLp{X_\theta}{p}
    \ge
    c\cdot \sqrt{p}\,L_K
    \ge
    \frac{c}{2}\cdot \sqrt{q}\,L_K.
\]
Now suppose $c_0\cdot n < q \le n$, and let $p_0$ be the largest element of $\mc D$. Then $p_0 \ge c_0\cdot n/2$, and hence
\[
    1 \le \frac{q}{p_0} \le \frac{2}{c_0}.
\]
Since $c_0$ is fixed, constants depending on $c_0$ are still absolute. Thus, by \cref{lem:2.2} and \cref{eq:3.5},
\[
    \rvLp{X_\theta}{q}
    \le
    C_{c_0}\cdot \rvLp{X_\theta}{p_0}
    \le
    C'_{c_0}\cdot \sqrt{p_0}\,L_K
    \le
     C'_{c_0}\cdot \sqrt{q}\,L_K.
\]
For the lower bound, monotonicity and \cref{eq:3.5} give
\[
    \rvLp{X_\theta}{q}
    \ge
    \rvLp{X_\theta}{p_0}
    \ge
    c\cdot \sqrt{p_0}\,L_K
    \ge
    c\cdot \sqrt{\frac{c_0\cdot q}{2}}\,L_K
    \ge
    c_{c_0}\cdot \sqrt{q}\,L_K.
\]
This proves the lemma.
\end{proof}

\subsection{\texorpdfstring{The endpoint $p>n$ and completion of the proof}{The endpoint p>n and completion of the proof}}\label{subsec:3.4}
The endpoint estimate is geometric and does not use isotropic position. It says that the support of a one-dimensional marginal is controlled by its $n$-th absolute moment, a consequence of the Brunn--Minkowski inequality. See also \cite[Lemma 3.6]{Paouris12b} for a more geometric formulation. We include the proof for completeness.

\begin{proposition}\label{prop:3.6}
Let $K \subset \R^n$ be a convex body of volume one and let $\theta \in \Sph$. Then
\[
    \sup_{1 \le p \le n} \frac{\margLp{K}{\theta}{p}}{\sqrt{p}}
    \le
    \sup_{p \ge 1} \frac{\margLp{K}{\theta}{p}}{\sqrt{p}}
    \le
    C
    \sup_{1 \le p \le n} \frac{\margLp{K}{\theta}{p}}{\sqrt{p}}.
\]
\end{proposition}

\begin{proof}
The lower estimate is immediate. If $n=1$, then $K$ is an interval of length one, so assume $n \ge 2$. Let $X$ be uniformly distributed on $K$, and write
\[
    X_\theta = \inner{X}{\theta}.
\]
After replacing $\theta$ by $-\theta$ if necessary, we may assume
\[
    R = \sup_{x\in K}\abs{\inner{x}{\theta}} = \sup \supp X_\theta.
\]
Since $R$ is the maximum absolute value of $\inner{\cdot}{\theta}$ on $K$, the support of $X_\theta$ is an interval
\[
    [m,R]
    \qquad\text{with}\qquad
    m \ge -R.
\]
Let $f$ be the density of $X_\theta$. By Brunn--Minkowski, in the form of \cite[Theorem~1]{Ball88}, $f^{1/(n-1)}$ is concave on its support. Thus
\[
    g = f^{1/(n-1)}
\]
is concave on $[m,R]$.

The following two midpoint estimates use only the concavity of $g$. For any $s \in [m,R]$, concavity gives
\[
    g\paren*{\frac{s+R}{2}}
    \ge
    \frac{g(s)+g(R)}{2}
    \ge
    \frac{g(s)}{2}.
\]
Thus
\[
    2 \int_0^R f(t)\,dt
    \ge
    \int_m^R f\paren*{\frac{s+R}{2}}\,ds
    \ge
    2^{-(n-1)} \int_m^R f(s)\,ds
    =
    2^{-(n-1)},
\]
and so,
\[
    \PP\paren{X_\theta \ge 0} = \int_0^R f(t)\,dt \ge 2^{-n}.
\]
Applying the same midpoint argument on the interval $[0,R]$, we obtain
\[
    2 \int_{R/2}^R f(t)\,dt
    =
    \int_0^R f\paren*{\frac{s+R}{2}}\,ds
    \ge
    2^{-(n-1)} \int_0^R f(s)\,ds.
\]
We can now conclude that
\[
    \PP\paren{X_\theta \ge R/2}
    =
    \int_{R/2}^R f(t)\,dt
    \ge
    2^{-n}\PP\paren{X_\theta \ge 0}
    \ge
    4^{-n},
\]
which then implies,
\[
    \E \abs{X_\theta}^n
    \ge
    \paren*{\frac{R}{2}}^n \PP\paren{X_\theta \ge R/2}
    \ge
    \paren*{\frac{R}{8}}^n.
\]
Rearranging terms we obtain the following bound,
\[
    R \le 8 \rvLp{X_\theta}{n}.
\]

Now, for $p > n$, since $\margLp{K}{\theta}{p} \le R$,
\[
    \frac{\margLp{K}{\theta}{p}}{\sqrt{p}}
    \le
    \frac{R}{\sqrt{n}}
    \le
    8 \frac{\margLp{K}{\theta}{n}}{\sqrt{n}}.
\]
We now conclude the proof by taking the supremum over $p > n,$
\[
    \sup_{p>n} \frac{\margLp{K}{\theta}{p}}{\sqrt{p}}
    \le
    8 \frac{\margLp{K}{\theta}{n}}{\sqrt{n}}.\qedhere
\]
\end{proof}

\begin{proof}[Proof of \cref{thm:1.1}]
Let $K \subset \R^n$ be a centered convex body of volume one, and let $\Sigma_K$ be the covariance matrix of the uniform measure on $K$. Put
\[
    T = (\det \Sigma_K)^{1/(2n)} \Sigma_K^{-1/2}
\]
and set
\[
    \widetilde K = T K.
\]
Then $T$ is symmetric and volume-preserving, and $\widetilde K$ is a centered isotropic convex body of volume one.

After adjusting if needed the absolute constant in the final result we may assume $c_0\cdot n \ge 1$. Apply \cref{lem:3.4} to the isotropic body \(\widetilde K\) with this linear map \(T\). We obtain vectors \(v_1,\dots,v_m\in\Sph\), with \(m=\lceil 9n/10\rceil\), such that \(Tv_1,\dots,Tv_m\) are pairwise orthogonal and each \(v_i\) satisfies \cref{eq:3.5} with \(K\) replaced by \(\widetilde K\). Since \(\widetilde K\) is isotropic and \(v_i\in\Sph\),
\[
    \margLp{\widetilde K}{v_i}{2}=L_{\widetilde K}.
\]
Applying \cref{lem:3.5} to each \(v_i\), followed by \cref{prop:3.6}, gives, for $1 \le p \le n$, 
\[
    c\cdot \sqrt{p}\,\margLp{\widetilde K}{v_i}{2}
    \le
    \margLp{\widetilde K}{v_i}{p}
    \le
    C\cdot \sqrt{p}\,\margLp{\widetilde K}{v_i}{2}
\]
\[
    \sup_{p \ge 1}
    \frac{\margLp{\widetilde K}{v_i}{p}}{\sqrt{p}}
    \le
    C\cdot \margLp{\widetilde K}{v_i}{2},
\]
and by the absolute-moment characterization of the \(\psi_2\) norm in \cref{subsec:2.2},
\[
    \snorm{\inner{\cdot}{v_i}}_{\psi_2(\widetilde K)}
    \le
    C\cdot \margLp{\widetilde K}{v_i}{2}.
\]

We now return to the original body. For $1 \le i \le m$, set
\[
    a_i = Tv_i.
\]
The vectors $a_1,\dots,a_m$ are pairwise orthogonal by construction. Moreover, \cref{lem:2.1} gives, for every $p \ge 1$,
\[
    \frac{\margLp{K}{a_i}{p}}{\margLp{K}{a_i}{2}}
    =
    \frac{\margLp{\widetilde K}{v_i}{p}}{\margLp{\widetilde K}{v_i}{2}},
\]
and the same identity with the $\psi_2$ norm in place of the $L^p$ norm. Hence, for $1 \le p \le n$ and $1 \le i \le m$,
\[
    c\cdot \sqrt{p}\,\margLp{K}{a_i}{2}
    \le
    \margLp{K}{a_i}{p}
    \le
    C\cdot \sqrt{p}\,\margLp{K}{a_i}{2}
\]
\[
    \sup_{p \ge 1}
    \frac{\margLp{K}{a_i}{p}}{\sqrt{p}}
    \le
    C\cdot \margLp{K}{a_i}{2},
\]
and
\[
    \snorm{\inner{\cdot}{a_i}}_{\psi_2(K)}
    \le
    C\cdot \margLp{K}{a_i}{2}.
\]

Normalize the orthogonal vectors by setting
\[
    \theta_i = \frac{a_i}{\abs{a_i}},
    \qquad
    \Theta = \set{\theta_1,\dots,\theta_m}.
\]
Then $\Theta$ is an orthonormal subset of $\Sph$ with $\abs{\Theta} \ge 9n/10$. The same estimates hold for $\theta_i$ by homogeneity. This proves the absolute moment assertion for $K$, and the tail formulation follows from the absolute moment estimates as noted after \cref{thm:1.1}.

If $K$ is already isotropic, then, for $1 \le i \le m$,
\[
    \margLp{K}{\theta_i}{2} = L_K,
\]
so the preceding display becomes, for $1 \le p \le n$ and $1 \le i \le m$,
\[
    c\cdot \sqrt{p}\,L_K
    \le
    \margLp{K}{\theta_i}{p}
    \le
    C\cdot \sqrt{p}\,L_K,
\]
and, for $1 \le i \le m$,
\[
    \snorm{\inner{\cdot}{\theta_i}}_{\psi_2(K)} \le C\cdot L_K.\qedhere
\]
\end{proof}

\section{\texorpdfstring{Remarks on obtaining an orthonormal basis for \cref{thm:1.1}}{Remarks on obtaining an orthonormal basis for Theorem 1.1}} \label{sec:4}
\cref{thm:1.1} does not guarantee an orthonormal basis. A tempting way to try to obtain one is to prove an analogue of \cref{thm:1.1} for arbitrary log-concave probability measures, choose one good direction, project the measure onto the orthogonal complement of that direction, and then iterate.

That analogue is false already in dimension one. Let $E$ be an exponential random variable with mean one and set $X=E-1$. Equivalently, $X$ has density
\[
    e^{-(x+1)} \one_{\set{x \ge -1}}.
\]
This is a centered, one-dimensional log-concave random variable with variance one. Moreover, for $t<1$,
\[
    \E e^{tX}
    =
    \int_{-1}^{\infty} e^{tx} e^{-(x+1)}\,dx
    =
    \frac{e^{-t}}{1-t},
\]
while, for $t \ge 1$,
\[
    \E e^{tX} = \infty.
\]
Hence, $X$ is not subgaussian, since a subgaussian random variable has a finite moment generating function for every real number.

This obstruction is also visible geometrically and explains why a purely greedy approach is generally dangerous. For example, as in \cite{Klartag08}, take a cone, 
\[
    K = \set{(x,s)\in \R^{n-1}\times \R : 0 \le s \le h,\ x \in (1-s/h)L}
\]
over an $(n-1)$-dimensional centered convex body $L$. Many good subgaussian directions may lie in the base hyperplane $\R^{n-1}\times\set{0}$ (depending of course, on the choice of $L$). However, if $S$ denotes the "height" marginal of a uniform point in $K$ in the $s$-direction, then $S$ has density
\[
    \frac{n}{h}(1-s/h)^{n-1}
    \qquad \text{for } 0 \le s \le h.
\]
Thus $nS/h$ converges in distribution to $E$. Translating $K$ so that it is centered replaces the height marginal by $S-\E S$, and since $\frac{n}{h}\E S=n/(n+1)$, we have
\[
    \frac{n}{h}(S-\E S) \to E-1.
\]
Thus the centered height marginal converges to the shifted exponential example above. A greedy procedure may therefore select $n-1$ good subgaussian directions in the hyperplane $\R^{n-1} \times \{0\}$ and leave, as the final orthogonal direction, the axis of the cone, which need not be subgaussian with a dimension-free constant.

Nevertheless, we believe that \cref{thm:1.1} could be strengthened to an orthonormal basis specifically when $K$ is an isotropic convex body.
\begin{conjecture}
    Let $K \subset \R^n$ be an isotropic convex body. There exists an orthonormal basis $\Theta \subset S^{n-1}$ such that for all $\theta \in \Theta$ one has
    \begin{equation*}
    c\cdot \sqrt{p}\,L_K
    \le
    \paren{\E \abs{\inner{X}{\theta}}^p}^{1/p}
    \le
    C\cdot \sqrt{p}\,L_K
\end{equation*}
where the upper bound holds for all $p \ge 1$ while the lower bound only holds for $1 \le p \le n$, where $c,C > 0$ are absolute constants.
\end{conjecture}
\bibliographystyle{amsplain0}
\bibliography{main}

\end{document}